\documentclass[letterpaper, 10 pt]{article}                            

\usepackage{graphicx}
\usepackage[textwidth=16cm]{geometry}
\usepackage{amssymb,amsmath,amsthm} 
\usepackage{microtype}
\usepackage{enumitem} 
\usepackage{cite}
\usepackage{xcolor}
\usepackage{bm}
\usepackage{bbm} 
\usepackage{subcaption}
\usepackage{algorithm}
\usepackage[noend]{algpseudocode}
\usepackage{authblk}
\usepackage[hidelinks]{hyperref}

\algnewcommand\algorithmicinput{\textbf{Initialization:}}
\algnewcommand\init{\item[\algorithmicinput]}
\algnewcommand\algorithmicevol{\textbf{Evolution:}}
\algnewcommand\evol{\item[\algorithmicevol]}
\algnewcommand\algorithmicawake{\textsf{\textit{{AWAKE}}}}
\algnewcommand\awake{\item[\algorithmicawake]}
\algnewcommand\algorithmicidle{\textsf{\textit{{IDLE}}}}
\algnewcommand\idle{\item[\algorithmicidle]}

\newcommand{\until}[1]{\{1,\ldots,#1\}}

\newcommand{\EE}{\mathcal{E}}
\newcommand{\GG}{\mathcal{G}}

\newcommand{\NN}{\mathcal{N}}

\newcommand{\mS}{\mathcal{S}}

\newcommand{\VV}{\mathcal{V}} 
\newcommand{\WW}{{W}}

\newcommand{\E}{\mathbb{E}}
\newcommand{\gs}{\textsl{g}}
\newcommand{\ff}{f_\textsl{best}}
\newcommand{\lxi}{x_{j{\mid}i}(t)}
\newcommand{\bG}{\bar{G}}

\newcommand{\m}{\mathop{\textrm{minimize}}}

\newcommand{\R}{\mathbb{R}}

\newcommand{\x}{\bm{x}}

\newcommand{\prox}{\text{prox}}

\newcommand{\NNio}{\NN_{i,out}}
\newcommand{\NNii}{\NN_{i,in}}
\newcommand{\tNNii}{\NN_{i,in}}

\newcommand{\DBP}{Distributed Block Proximal Method }
\newcommand{\DBPnospace}{Distributed Block Proximal Method}

\newcommand{\bx}{\bar{x}}

\newcommand{\bsigma}{m}
\newcommand{\bS}{\bar{S}}
\newcommand{\bR}{\bar{R}}

\makeatletter
\newcommand{\pushright}[1]{\ifmeasuring@#1\else\omit\hfill$\displaystyle#1$\fi\ignorespaces}
\newcommand{\pushleft}[1]{\ifmeasuring@#1\else\omit$\displaystyle#1$\hfill\fi\ignorespaces}
\makeatother

\newcommand\oprocendsymbol{\hbox{$\square$}}
\newcommand\oprocend{\relax\ifmmode\else\unskip\hfill\fi\oprocendsymbol}

\theoremstyle{plain}
\newtheorem{theorem}{Theorem}

\newtheorem{lemma}{Lemma}

\theoremstyle{definition}
\newtheorem{assumption}{Assumption}

\theoremstyle{remark}
\newtheorem{remark}{Remark}

\title{On the Linear Convergence Rate\\ of the Distributed Block Proximal Method
\thanks{
This result is part of a project that has received funding from the European Research Council (ERC) under the European Union's Horizon 2020 research and innovation programme (grant agreement No 638992 - OPT4SMART).\newline\newline
\textcopyright  2020 IEEE.  Personal use of this material is permitted.  Permission from IEEE must be obtained for all other uses, in any current or future media, including reprinting/republishing this material for advertising or promotional purposes, creating new collective works, for resale or redistribution to servers or lists, or reuse of any copyrighted component of this work in other works.\newline\newline
Digital Object Identifier 10.1109/LCSYS.2020.2976311}}

\author[]{Francesco~Farina}
\author[]{Giuseppe~Notarstefano}
\affil[]{Department of Electrical, Electronic and Information Engineering\\ Universit{\`a} di Bologna, Bologna, Italy.}

\date{}

\begin{document}
\maketitle

\begin{abstract}
  The recently developed \DBPnospace, for solving stochastic big-data convex
  optimization problems, is studied in this paper under the assumption of
  constant stepsizes and strongly convex (possibly non-smooth) local objective functions.
  This class of problems arises in many learning and classification problems in
  which, for example, strongly-convex regularizing functions are included in the
  objective function, the decision variable is extremely high dimensional, and
  large datasets are employed.
  The algorithm produces local estimates by means of block-wise updates and communication among the agents. 
  The expected distance from the
  (global) optimum, in terms of cost value, is shown to decay linearly to a
  constant value which is proportional to the selected local stepsizes.
    A numerical example involving a classification problem corroborates the theoretical results.
\end{abstract}


\section{Introduction} 
In this paper, we address in a distributed way stochastic big-data convex
optimization problems involving \emph{strongly convex} (possibly nonsmooth)
local objective functions, by means of the
\DBPnospace~\cite{farina2019randomized,farina2019subgradient}.
Problems with this structure naturally arise in many learning and control
problems in which the decision variable is extremely high dimensional and large
datasets are employed.
Relevant examples include: direct policy search in reinforcement
learning~\cite{recht2019reinforcement}, dynamic problems involving stochastic
functions generated from collected samples to be processed
online~\cite{xiao2010dual}, learning problems involving massive
datasets in which sample average approximation techniques are
used~\cite{kleywegt2002sample}, and settings in which only noisy subgradients of the
objective functions can be computed at each time
instant~\cite{ram2010distributed}.

Distributed algorithms for solving stochastic problems have been widely
studied~\cite{ram2010distributed,agarwal2011distributed,srivastava2011distributed,nedic2016stochastic,li2018stochastic,ying2018performance}. On
the other side, distributed algorithms for big-data problems
through block communication have started to appear only
recently~\cite{necoara2013random,arablouei2013distributed,wang2017coordinate,notarnicola2018distributed, FARINA2019243}. The
\DBP solves problems that can be together non-smooth, stochastic and big-data, thus distinguishing from the above works (see~\cite{farina2019randomized} for a comprehensive literature review). This
algorithm evolves through block-wise communication and updates (involving subgradients of the local functions and proximal mappings induced by some distance genereting functions) and it has been already shown to achieve a sublinear convergence rate on
problems with non-smooth convex objective functions. 
The contribution of this
paper is to extend this result by showing that, under strongly convex (possibly non-smooth) local objective
functions and constant stepsizes, the \DBP exhibits a linear convergence rate (with a constant error
term) to the optimal cost in expected value. 
The main challenge in the linear-rate analysis relies in the block-wise nature of the algorithm.

\section{Set-up and preliminaries}\label{sec:problem}
\subsection{Notation, definitions and preliminary results}
Given a vector $x\in\R^n$, we denote by $x_\ell$ the
$\ell$-th block of $x$, i.e., given a partition of the identity matrix
$I=[U_1,\dots,U_B]$, with $U_\ell\in\R^{n\times n_\ell}$ for all $\ell$ and
$\sum_{\ell=1}^B n_\ell=n$, it holds $x = \sum_{\ell=1}^B U_\ell x_\ell$
and $x_\ell=(U_\ell)^\top x$. In general, given a vector $x_i\in\R^n$, we denote by $x_{i,\ell}$ the $\ell$-th block of $x_i$. 
Given a matrix $A$, we denote by
$a_{ij}$ the element of $A$ located at row $i$ and column $j$. 
Given two vectors $a,b\in\R^n$ we denote by $\langle a,b\rangle$ their scalar product. Given a discrete random variable
$r\in\until{R}$, we denote by $P(r=\bar{r})$ the probability of $r$ to be equal
to $\bar{r}$ for all $\bar{r}\in\until{R}$. Given a nonsmooth function $f$, we
denote by $\partial f(x)$ its subdifferential at $x$.

The following preliminary result will be used in the paper.
\begin{lemma}\label{lemma:series}
    Given any two scalars $\delta\neq\gamma\neq 1$, it holds that
    \begin{enumerate}[label=(\roman*)]
        \item\label{item:1} $\sum_{s=r}^t\delta^s = \frac{\delta^r-\delta^{t+1}}{1-\delta}$
        \item\label{item:3} $\sum_{s=0}^t \delta^{t-s}\gamma^s=\frac{\delta^{t+1}-\gamma^{t+1}}{\delta-\gamma}.$\oprocend
    \end{enumerate}
\end{lemma}

\subsection{Distributed stochastic optimization set-up}
Let us start by formalizing the optimization problem addressed in this paper.
We consider problems in the form
\begin{equation}\label{pb:problem}
    \begin{aligned}
        & \m_{x\in X}
        & & \sum_{i=1}^N \E[h_i(x;\xi_i)].
    \end{aligned}
\end{equation}
where $\xi_i$ is a random variable and $x\in\R^n$, $n\gg 1$, has a block structure, i.e., $x=[x_1^\top,\dots,x_B^\top]^\top.
$, with $x_\ell\in\R^{n_\ell}$ for all $\ell$ and $\sum_{\ell=1}^B n_\ell = n$. The decision variable $x$ can be very high-dimensional, which calls for block-wise algorithms.

Let $f_i(x)\triangleq\E[h_i(x;\xi_i)]$. Moreover, let $g_{i}(x;\xi_i)\in\partial h_i(x;\xi_i)$ (resp. $\gs_i(x)\in\partial f_i(x)$) be a subgradient of $h_i(x;\xi_i)$ (resp. $f_i(x)$) computed at $x$. Then, the following assumption holds for problem~\eqref{pb:problem}.
\begin{assumption}[Problem structure]\label{assumption:problem_structure}
    \hspace{1ex}
    \begin{enumerate}[label=(\Alph*)]
    \item The constraint set $X$ has the block structure $X = X_1\times \dots\times X_B$, where, for $\ell=1,\dots,B$, the set $X_\ell\subseteq\R^{n_\ell}$ is closed and convex, and $\sum_{\ell=1}^B n_\ell = n$.
    \item\label{assumption:strong_convexity} The function $h_i(x,\xi_i):\R^n\to\R$ is continuous, strongly convex and possibly nonsmooth for all $x\in X$ and every $\xi_i$, for all $i\in\until{N}$. In particular, there exists a constant $\bsigma>0$ such that
        $f_i(a)\geq f_i(b) - \langle \gs_i(b), b-a\rangle + \frac{\bsigma}{2}\| a- b \|^2$,
    for all $a,b\in X$ and all $i\in\until{N}$.

    \item\label{subg} every subgradient $g_{i}(x;\xi_i)$ is an unbiased estimator of the subgradient of $f_i$, i.e.,
    $\E[g_{i}(x;\xi_i)]=\gs_i(x)$.
    Moreover, there exist constants $G_{i}\in[0,\infty)$ and $\bG_{i}\in[0,\infty)$
    such that 
    $\E[\|g_i(x;\xi_i)\|]\leq G_{i}$, and $\E[\|g_i(x;\xi_i)\|^2]\leq \bG_{i}$,
    for all $x$ and $\xi_i$, for all $i\in\until{N}$.\oprocend
    \end{enumerate}
\end{assumption}

Let us denote by $g_{i,\ell}(x;\xi_i)$ the $\ell$-th block of $g_{i}(x;\xi_i)$
and let $\gs(x)\in\partial f(x)$ be a subgradient of $f$ computed at $x$. Then,
Assumption~\ref{assumption:problem_structure}\ref{subg} implies that
$\E[\|g_{i,\ell}(x;\xi_i)\|]\leq G_i$ for all $\ell$ and $\|\gs_i(x)\|\leq G_i$.
Moreover, let $\bG\triangleq\sum_{i=1}^N \bG_i$ and
$G\triangleq\sum_{i=1}^N G_i$. Then, $\|\gs(x)\|{\leq} G$ and
$\|\gs_i(x)\|{\leq} G$ for all $i$.

\vspace{2ex}

Problem~\eqref{pb:problem} is to be cooperatively solved by a network of $N$
agents. Agents locally know only a portion of the entire optimization
problem. Namely, agent $i$ knows only $g_i(x;\xi_i)$ for any $x$ and
  $\xi_i$, and the constraint set $X$.  The communication network is assumed to
satisfy the next assumption.
\begin{assumption}[Communication structure]\label{assumption:communication}
    \hspace{1ex}
    \begin{enumerate}[label=(\Alph*)]
        \item The network is modeled through a weighted \emph{strongly connected} directed graph $\GG=(\VV,\EE, \WW)$ with $\VV=\until{N}$, $\EE\subseteq\VV\times\VV$ and $\WW\in\R^{N\times N}$ being the weighted adjacency matrix. We define $\NNio\triangleq\{j\mid (i,j)\in\EE\}\cup\{i\}$ and $\NNii\triangleq\{j\mid (j,i)\in\EE\}\cup\{i\}$.
        \item For all $i,j\in\until{N}$, the weights $w_{ij}$ of the weight matrix $\WW$ satisfy
        \begin{enumerate}[label=(\roman*)]
            \item $w_{ij}>0$ if and only if $j\in\NNii$;
            \item there exists a constant $\eta>0$ such that
                $w_{ii}\geq\eta$ and if $w_{ij}>0$, then $w_{ij}\geq\eta$;
            \item $\sum_{j=1}^N w_{ij}=1$ and $\sum_{i=1}^N w_{ij}=1$.\oprocend
        \end{enumerate} 
    \end{enumerate}
\end{assumption}

In order to solve problem~\eqref{pb:problem} agents will be using ad-hoc \emph{proximal mappings} (see, e.g.,~\cite{dang2015stochastic}). In particular, a function $\omega_{\ell}$ is associated to the
$\ell$-th block of the optimization variable for all $\ell$. Let the function
$\omega_{\ell}:X_\ell\to\R$, be continuously differentiable and
$\sigma_{\ell}$-strongly convex. Functions $\omega_{\ell}$ are sometimes
referred to as distance generating functions. Then, we define the \emph{Bregman's divergence} associated to $\omega_{\ell}$ as
\begin{equation*}
    \nu_{\ell}(a,b)=\omega_{\ell}(b)-\omega_{\ell}(a)-\langle \nabla \omega_{\ell} (a), b-a\rangle,
\end{equation*}
for all $a,b\in X_\ell$. Moreover, given $a\in X_\ell$, $b\in\R^{n_\ell}$
and $c\in\R$, the proximal mapping associated to $\nu_{\ell}$ is defined as
\begin{equation}\label{eq:prox}
    \prox_{\ell} (a,b,c)=\arg\min_{u\in X_\ell}\Bigg(\langle b, u \rangle+\frac{1}{c} \nu_{\ell}(a, u)\Bigg).
\end{equation}
We make the following assumption on the functions $\nu_{\ell}$.
\begin{assumption}[Bregman's divergences properties]\label{assumption:proximal_functions}
    \hspace{1ex}
    \begin{enumerate}[label=(\Alph*)]
        \item\label{assumption:quadratic_growth} There exists a constant $Q>0$ such that 
        \begin{equation}\label{eq:assumption_growth}
            \nu_{\ell}(a,b)\leq\frac{Q}{2}\|a-b\|^2,\quad\forall a,b\in X_\ell
        \end{equation}
        for all $\ell\in\until{B}$.
        \item\label{assumption:separate} For all $\ell\in\until{B}$, the function $\nu_{\ell}$ satisfies
        \begin{equation}\label{eq:separable}
            \hspace{-3ex}\nu_{\ell}\left(\sum_{j=1}^N \theta_j a_j,b\right){\leq} \sum_{j=1}^N \theta_j \nu_{\ell}(a_j,b), \; \forall a_1,\dots,a_N,b{\in} X_\ell,
        \end{equation}
        where $\sum_{j=1}^N \theta_j=1$ and $\theta_j\geq 0$ for all $j$.\oprocend
    \end{enumerate}
\end{assumption}
Notice that Assumption~\ref{assumption:proximal_functions}\ref{assumption:quadratic_growth} implies that, given any two points $a,b\in X$, 
\begin{equation}
    \sum_{\ell=1}^B \nu_{\ell}(a_\ell,b_\ell)\leq\frac{Q}{2}\sum_{\ell=1}^B\|a_\ell-b_\ell\|^2=\frac{Q}{2}\|a-b\|^2.
\end{equation}
Moreover, Assumption~\ref{assumption:proximal_functions}\ref{assumption:separate} is satisfied by many functions (such as the
quadratic function and the exponential function)
and conditions on $\omega_{\ell}$
guaranteeing~\eqref{eq:separable} can be provided~\cite{bauschke2001joint}.

\section{Distributed Block Proximal Method}\label{sec:algorithm}
Let us now recall the \DBP for solving problem~\eqref{pb:problem} in a distributed way. The pseudocode of the algorithm is reported in Algorithm~\ref{alg:DBS}, where, for notational convenience, we defined $g_{i,\ell}(t)\triangleq g_{i,\ell}(y_i(t);\xi_i(t))$. We refer to~\cite{farina2019randomized} for all the details.

\begin{algorithm}
    \small
    \begin{algorithmic}
        \init $x_i(0)$
		\evol for $t=0,1,\dots$
         
        \State \textsc{Update} for all $j\in\NNii$
        \begin{equation*}\label{eq:xl_update}
            x_{j,\ell{\mid}i}(t) = 
            \begin{cases}
                x_{j,\ell}(t), &\text{if }\ell=\ell_j(t-1) \; \text{and} \; s_j(t-1)=1\\
                x_{j,\ell{\mid}i}(t-1), &\text{otherwise}
            \end{cases}
        \end{equation*}
        \If{$s_i^t=1$ }
        \State\textsc{Pick} $\ell_i(t)\in\until{B}$ with $P(\ell_i(t)=\ell)=p_{i,\ell}>0$
        \State\textsc{Compute} 
		\begin{equation*}\label{eq:y_update_l}
			y_i(t) = \sum_{j\in\tNNii} w_{ij} \lxi
		\end{equation*}
        \State\textsc{Update} 
        \begin{equation*}\label{eq:x_update_l}
            x_{i,\ell}(t+1) = 
            \begin{cases}
                \prox_{\ell}\bigg(y_{i,\ell}(t),g_{i,\ell}(t),\alpha_i\bigg),&\text{if } \ell=\ell_i(t)\\
                x_{i,\ell}(t),&\text{otherwise}
            \end{cases}
        \end{equation*}
        \State\textsc{Broadcast} $x_{i,\ell_i(t)}(t+1)$ to $j\in\NNio$
        \Else{ $x_i(t+1)=x_i(t)$}
        \EndIf
		
	\end{algorithmic}
	\caption{\DBP}\label{alg:DBS}
\end{algorithm} 

The algorithm works as follows. Each
agent $i$ maintains a local solution estimate $x_i(t)$ and a local copy of the
estimates of its in-neighbors (namely, $\lxi$ denotes the copy of the solution
estimate of agent $j$ at agent
$i$).
The initial conditions are initialized with random
(bounded) values $x_i(0)$ which are shared between neighbors.
At each iteration, agents can be awake or idle, thus modeling a possible asynchrony in the network. The probability of agent $i$ to be awake is denoted by $p_{i,on}\in (0,1]$.
If agent $i$ is awake at iteration $t$, it picks randomly a block $\ell_i(t)\in\until{B}$, some $\xi_i(t)$, and performs two updates:
\begin{enumerate}[label=(\roman*)]
    \item it computes a weighted average of its in-neighbors' estimates $\lxi$, $j\in\tNNii$;
    \item it computes
      $x_i(t+1)$ by updating the $\ell_i(t)$-th block of $x_i(t)$ through a proximal mapping step (with a constant stepsize $\alpha_i$) and leaving the other blocks unchanged.
\end{enumerate}
Finally, it broadcasts $x_{i,\ell_i(t)}(t+1)$ to its out-neighbors. The
status (awake or idle) of node $i$ at iteration $t$ is modeled as a random
variable $s_i(t)\in\{0,1\}$ which is $1$ with
probability $p_{i,on}$ and $0$ with probability $1-p_{i,on}$.

As already stated in~\cite{farina2019randomized}, it is worth remarking that all the quantities involved in the \DBP are local for each node. In fact, each node has locally defined probabilities (both of awakening and block drawing) and local stepsizes.
Moreover, it is worth recalling that, from~\cite[Lemma~5]{farina2019randomized}, we have $x_{j{\mid}i}(t)=x_j(t)$ for all $t$ and hence, Algorithm~\ref{alg:DBS} can be compactly rewritten as follows. For all $i\in\until{N}$ and all $t$, if $s_i(t)=1$,
\begin{align}
    y_i(t) &= \sum_{j=1}^N w_{ij} x_j(t),\label{eq:y_update}\\
    x_{i,\ell}(t+1) &= 
    \begin{cases}
        \textup{prox}_{\ell}\bigg(y_{i,\ell}(t),g_{i,\ell}(t),\alpha_i\bigg),&\text{if } \ell=\ell_i(t),\\
        x_{i,\ell}(t),&\text{otherwise},\label{eq:x_update}
    \end{cases}
\end{align}
else, $x_i(t+1)=x_i(t)$.
We will use~\eqref{eq:y_update}-\eqref{eq:x_update} in place of Algorithm~\ref{alg:DBS}, in the following analysis.

\section{Algorithm analysis and convergence rate}\label{sec:analysis}
Let $\x(\tau)\triangleq[x_1(\tau)^\top, \dots, x_N(\tau)^\top]^\top$ and let $\mS(t)\triangleq \{\x(\tau)\mid \tau\in\{0,\dots,t\}\}$ be the se set of estimates generated by the \DBP up to iteration $t$. Moreover, define the probability of node $i$ to both be awake and pick block $\ell$ as $$\pi_{i,\ell}\triangleq p_{i,on} p_{i,\ell}$$ and define $a\triangleq[\alpha_1,\dots,\alpha_N]^\top$, $a_M\triangleq\max_{i}\alpha_i$ and $a_m\triangleq\min_{i}\alpha_i$. Moreover, define the average (over the agents) of the local estimates at $t$ as
\begin{equation}\label{eq:x_mean}
    \bar{x}(t)\triangleq\frac{1}{N}\sum_{i=1}^N x_i(t).
\end{equation}
Finally, let us make the following assumption about the random variables involved in the algorithm.
\begin{assumption}[Random variables]\label{assumption:random_variables}
    \hspace{1ex}
    \begin{enumerate}[label=(\Alph*)]
    \item\label{assumption:iid} The random variables $\ell_i(t)$ and $s_i(t)$ are independent and identically distributed for all $t$, for all $i\in\until{N}$.
    \item\label{assumption:indepentent}
    For any given $t$, the random variables $s_i(t)$, $\ell_i(t)$ and $\xi_i(t)$ are independent of each other for all $i\in\until{N}$.
    \item\label{assumption:initial}
    There exists a constants $C_i\in[0,\infty)$ such that $\E[\|x_i(0)\|]\leq C_i$ for all $i\in\until{N}$ and hence $\E[\|\x(0)\|]\leq C=\sum_{i=1}^N C_i$. \oprocend
    \end{enumerate}
\end{assumption}

In the following we analyze the convergence properties of the \DBP with constant
stepsizes under the previous assumptions. We start by showing that consensus is
achieve in the network, by specializing the results
in~\cite{farina2019randomized}. Then, we show that also optimality is achieved
in expected value and with a constant error, by studying the properties of an
ad-hoc Lyapunov-like function. Finally, we show how the main result implies a
linear convergence rate for the algorithm.

\subsection{Reaching consensus}
The following lemma characterizes the expected distance of $x_i(t)$ and $y_i(t)$ from the average $\bx(t)$ (defined in~\eqref{eq:x_mean}).
\begin{lemma}\label{lemma:xi-bx}
  Let Assumptions~\ref{assumption:problem_structure},~\ref{assumption:communication},~\ref{assumption:random_variables} hold. Then, there exist constants $M\in(0,\infty)$
  and $\mu_M\in(0,1)$ such that
\begin{align}
    \E[\|x_i(t)-\bx(t)\|] &\leq \mu_M^{t-1} \bR+\bS,\label{eq:xi-bx}\\
    \E[\|y_i(t)-x_i(t)\|] &\leq 2 \mu_M^{t-1} \bR + 2\bS
\end{align}
for all $i\in\until{N}$ and all $t\geq 1$, with
$\bR= MB\left(C- \frac{a_M G}{\sigma(1-\mu_M)}\right)$ and $\bS= a_M \frac{MBG}{\sigma}\frac{2-\mu_M}{1-\mu_M}$.
\end{lemma}
\begin{proof}
    The proof follows by using constant stepsizes in~\cite[Lemma~7 and Lemma~8]{farina2019randomized}.
\end{proof}

In the next section, in order to prove the convergence to the optimal cost with a linear rate, we will need the following result assuring the boundedness of a particular quantity. In particular, given a scalar $c\in(0,1)$, let us define
\begin{equation}
    \beta(t)\triangleq\sum_{\tau=0}^t c^{t-\tau} \E[\|x_i(\tau)-\bx(\tau)\|].
\end{equation}
Then, the next lemma provides a bound on $\beta(t)$ for all $t$.

\begin{lemma}\label{lemma:x_constant}
    Let Assumptions~\ref{assumption:problem_structure},~\ref{assumption:communication},~\ref{assumption:random_variables} hold.  Then, for any scalar $c\in (0,1)$, 
    \begin{enumerate}[label=(\roman*)]
        \item if $c\neq\mu_M$,
        \begin{align}
            \beta(t)&\leq  c^{t}\left(C + \frac{\bR}{c-\mu_M}\right) +\frac{1-c^t}{1-c}\bS \label{eq:xbound_sub}
        \end{align}
        \item if $c=\mu_M$, 
        \begin{align}
            \beta(t)&\leq  c^{t}\left(C + \frac{t\bR}{c}\right) +\frac{1-c^t}{1-c}\bS\label{eq:xbound_sub_eq}
        \end{align}
    \end{enumerate}
    for all $i\in\until{N}$, for all $t$.
\end{lemma}
\begin{proof}
    By using Assumption~\ref{assumption:random_variables}\ref{assumption:initial}, for $\tau=0$, one has
    \begin{align}
        \E&[\|x_i(0)-\bx(0)\|]\leq  \E[\|x_i(0)\|] +\E[\|\bx(0)\|]\nonumber\\
        &\leq C_i +\frac{1}{N}\sum_{j=1}^N C_j \leq C_i +\max_j C_j \leq C\label{eq:x0}
    \end{align}
    Hence, $\beta(t)\leq c^t C +\sum_{\tau=1}^t  c^{t-\tau}\E[\|x_i(\tau)-\bx(\tau)\|]$ and, from Lemma~\ref{lemma:xi-bx}, we have 
    \begin{align}
        \beta(t)\leq c^t C + \bR \sum_{\tau=1}^t c^{t-\tau}\mu_M^{\tau-1}+\bS\sum_{\tau=1}^t c^{t-\tau}\label{eq:start}
    \end{align}
    Let us consider the case $c\neq \mu_M$. By using Lemma~\ref{lemma:series}, one easily gets
    \begin{align*}
        \beta(t)&\leq  c^{t}C + \frac{c^t-\mu_M^t}{c-\mu_M}\bR +\frac{1-c^t}{1-c}\bS\\
        &\leq  c^{t}\left(C + \frac{\bR}{c-\mu_M}\right) +\frac{1-c^t}{1-c}\bS
    \end{align*}
    where in the second line we have removed the negative term depending on $\mu_M^t$.
    For the case $c= \mu_M$ we have
    \begin{equation}
        \sum_{\tau=1}^t c^{t-\tau}\mu_M^{\tau-1} = \sum_{\tau=1}^{t} c^{t-1} = t c^{t-1}\label{eq:p1_eq}
    \end{equation}
    and~\eqref{eq:xbound_sub_eq} is obtained by substituting~\eqref{eq:p1_eq} in~\eqref{eq:start} and using Lemma~\ref{lemma:series}.
\end{proof}

\subsection{Reaching optimality}
Let us start by defining a Lyapunov-like function
\begin{equation}\label{eq:V}
    V_i^\tau \triangleq \sum_{\ell=1}^B \pi_{i,\ell}^{-1}\nu_{\ell}(x_{i,\ell}^{\tau},x_{\ell}^\star)
\end{equation}
and let $V^t \triangleq \sum_{i=1}^N V_i^t$.
Moreover, define
\begin{equation}
    \ff(\bx^t)\triangleq\min_{\tau\leq t} \E[f(\bar{x}^\tau)]
\end{equation}
and $\pi_m=\min_{i,\ell}\pi_{i,\ell}$.
Then, the following result holds true and will be the key for proving the linear convergence rate of the \DBP under the previous assumptions.
\begin{lemma}\label{lemma:strongly}
    Let Assumptions~\ref{assumption:problem_structure},~\ref{assumption:communication},~\ref{assumption:proximal_functions} and~\ref{assumption:random_variables} hold. Moreover, let $\alpha_i\leq\frac{Q}{\bsigma}$ for all $i$. Then, for all $t$,
    \begin{align}
        \E[V(t&+1)]\leq \left(1-\frac{\bsigma a_m \pi_m}{Q}\right)\E[V(t)]- \sum_{i=1}^N \alpha_i\left( \E[f_i(y_i(t))] - f_i(x^\star)\right) + \frac{a_M^2 \bG}{2\sigma}.\label{eq:lemma_strongly}
    \end{align}
\end{lemma}
\begin{proof}
    In order to simplify the notation, let us denote $\gs_i(t)=\gs_i(y_i(t))$.
    By using the same arguments used in the proof of~\cite[Theorem~1]{farina2019randomized} we have
\begin{align}
    &\E[V_i(t+1)\mid\mS(t)]\leq V_i(t) - \sum_{\ell=1}^B\nu_{\ell}(x_{i,\ell}(t),x_{\ell}^\star) + \sum_{\ell=1}^B\nu_{\ell}(y_{i,\ell}(t),x_{\ell}^\star) -\alpha_i  \langle \gs_i(t),y_i(t)-x^\star\rangle+\frac{\alpha_i^2 \bG_i}{2\sigma}
    \label{eq:start2}
\end{align} 
    Now, By exploiting Assumptions~\ref{assumption:problem_structure}\ref{assumption:strong_convexity},~\ref{assumption:proximal_functions}\ref{assumption:quadratic_growth} and~\eqref{eq:assumption_growth}, one has that, for all $t$,
    \begin{align}
        \alpha_i  \langle \gs_i(t),y_i(t)-x^\star\rangle &\geq \alpha_i \left( f_i(y_i(t))-f_i(x^\star) + \frac{\bsigma}{2}\|y_i(t)-x^\star\|^2  \right)\nonumber\\
        &\geq \alpha_i \left( f_i(y_i(t))-f_i(x^\star)\right) +
          \frac{\bsigma\alpha_i}{Q}\sum_{\ell=1}^B\nu_\ell(y_{i,\ell}(t),x_{\ell}^\star). \label{eq:p1}
    \end{align}
     Now, by using~\eqref{eq:p1} in~\eqref{eq:start2} and by exploiting the fact that $\alpha_i\leq\frac{Q}{\bsigma}$, we get
    \begin{align}
        \E[V_i(t+1)\mid\mS(t)]
        &\leq V_i(t) - \sum_{\ell=1}^B\nu_\ell(x_{i,\ell}(t),x_{\ell}^\star) + \left(1-\frac{\bsigma\alpha_i}{Q}\right)\sum_{\ell=1}^B\nu_\ell(y_{i,\ell}(t),x_{\ell}^\star)\nonumber\\
        &\hspace{3ex}-\alpha_i   \left( f_i(y_i(t))-f_i(x^\star)\right)+\frac{\alpha_i^2 \bG_i}{2\sigma}\nonumber\\
        &\leq V_i(t) - \sum_{\ell=1}^B\nu_\ell(x_{i,\ell}(t),x_{\ell}^\star) + \left(1-\frac{\bsigma\alpha_i}{Q}\right)\sum_{j=1}^N w_{ij}\sum_{\ell=1}^B\nu_\ell(x_{j,\ell}(t),x_{\ell}^\star)\nonumber\\
        &\hspace{3ex}-\alpha_i   \left( f_i(y_i(t))-f_i(x^\star)\right)+\frac{\alpha_i^2 \bG_i}{2\sigma},
    \end{align}
    where in the second inequality we used assumption~\ref{assumption:proximal_functions}\ref{assumption:separate}.
    If we now sum over $i$, by noticing that $a_m\leq\alpha_i$ for all $i$, we obtain
    \begin{align}
        \sum_{i=1}^N\E[V_i^{t+1}\mid\mS(t)]\leq &\sum_{i=1}^N V_i(t) - \sum_{i=1}^N \sum_{\ell=1}^B\nu_\ell(x_{i,\ell}(t),x_{\ell}^\star) \nonumber\\
        &+ \sum_{i=1}^N \left(1-\frac{\bsigma\alpha_i}{Q}\right)\sum_{j=1}^N w_{ij}\sum_{\ell=1}^B\nu_\ell(x_{j,\ell}(t),x_{\ell}^\star)\nonumber\\
        &-\sum_{i=1}^N\alpha_i   \left( f_i(y_i(t))-f_i(x^\star)\right)+\sum_{i=1}^N\frac{\alpha_i^2 \bG_i}{2\sigma}.
    \end{align}
    Now, by using the fact that $a_m\leq \alpha_i\leq a_M$ for all $i$, the double stochasticity of $W$ from Assumption~\ref{assumption:communication}, and the definition of $\bar{G}$, one easily obtains that 
    \begin{align}
        \sum_{i=1}^N\E[V_i^{t+1}\mid\mS(t)]
        &\leq \sum_{i=1}^N V_i(t) -\frac{\bsigma a_m}{Q}\sum_{i=1}^N \sum_{\ell=1}^B\nu_\ell(x_{i,\ell}(t),x_{\ell}^\star)-\sum_{i=1}^N\alpha_i   \left( f_i(y_i(t))-f_i(x^\star)\right)+\frac{a_M^2 \bG}{2\sigma}.\label{eq:V_part}
    \end{align}
    Moreover, by using~\eqref{eq:V} we can rewrite
    \begin{align}
        \sum_{i=1}^N V_i(t) -\frac{\bsigma a_m}{Q}\sum_{i=1}^N \sum_{\ell=1}^B\nu_{\ell}(x_{i,\ell}(t),x_{\ell}^\star)
        &= \sum_{i=1}^N \sum_{\ell=1}^B \left(\pi_{i,\ell}^{-1}\nu_{\ell}(x_{i,\ell}(t),x_{\ell}^\star) -\frac{\bsigma a_m}{Q}\nu_{\ell}(x_{i,\ell}(t),x_{\ell}^\star)\right)\nonumber\\
        &\leq\left(1- \frac{\bsigma \pi_m  a_m}{Q}\right) V(t)\label{eq:Vdec}
    \end{align}
    where we have used the fact that $\sum_{\ell=1}^B\pi_m^{-1}\nu_{\ell}(a,b)\geq\sum_{\ell=1}^B\pi_{i,\ell}^{-1}\nu_{\ell}(a,b)$. Finally, by plugging~\eqref{eq:Vdec} in~\eqref{eq:V_part} and by using tower property of conditional expectation one gets~\eqref{eq:lemma_strongly}.
\end{proof}

Thanks to the previous results, we are now ready to state and prove the main result of this paper.
\begin{theorem}\label{theorem:bound}
    Let Assumptions~\ref{assumption:problem_structure},~\ref{assumption:communication},~\ref{assumption:proximal_functions} and~\ref{assumption:random_variables} hold. Moreover, let $\alpha_i\leq\frac{Q}{\bsigma}$ for all $i$ and let $c \triangleq \left(1-\frac{\bsigma a_m \pi_m}{Q}\right)$.
    Then, 
    \begin{enumerate}
        \item if $c\neq\mu_M$,
    \begin{align}
        \ff(\bx^t)-f(x^\star)&\leq\frac{c^{t}}{1-c^{t+1}}(Q+R_1)+S,\label{eq:f_bound}
    \end{align}
    \item if $c=\mu_M$, 
    \begin{align}
        \ff(\bx^t)-f(x^\star)&\leq \frac{c^{t}}{1-c^{t+1}} \left(Q +tR_2\right) +S,\label{eq:f_bound_eq}
    \end{align}
    \end{enumerate}
    where $Q=(1-c)\left(\frac{\E[V^0]}{a_m} + 3GC\right)$, $R_1= \frac{(1-c)3G\bR}{c-\mu_M}$, $R_2= \frac{(1-c)3G\bR}{c}$. and $S=\frac{ a_M^2 \bG}{2\sigma a_m} + 3G\bS$.
\end{theorem}
\begin{proof}
    By recursively applying~\eqref{eq:lemma_strongly}, one has
    \begin{align}
        \sum_{\tau=0}^{t}c^{t-\tau}&\sum_{i=1}^N \alpha_i\left( \E[f_i(y_i(\tau))] - f_i(x^\star)\right)
        \leq c^{t+1} \E[V^0] + \sum_{\tau=0}^{t}c^{t-\tau} \frac{a_M^2 \bG}{2\sigma}\nonumber
    \end{align}
    Moreover, since $a_m\leq \alpha_i$ for all $i$, 
    \begin{align}
        \sum_{\tau=0}^{t}c^{t-\tau} a_m \sum_{i=1}^N\left( \E[f_i(y_i(\tau))] - f_i(x^\star)\right) 
        &\leq\sum_{\tau=0}^{t}c^{t-\tau}\sum_{i=1}^N \alpha_i\left( \E[f_i(y_i(\tau))] - f_i(x^\star)\right)\nonumber\\
        &\leq c^{t+1} \E[V^0]  + \sum_{\tau=0}^{t}c^{t-\tau} \frac{a_M^2 \bG}{2\sigma}\nonumber\\
        &= c^{t+1} \E[V^0]  + \frac{a_M^2 \bG}{2\sigma}\frac{1-c^{t+1}}{1-c}\label{eq:fsum}
    \end{align}
    where in the last line we used Lemma~\ref{lemma:series}, thanks to the fact that since by assumption $\alpha_i\leq\frac{Q}{\bsigma}$, we have $c\in(0,1)$. Then, from the convexity of $f$ we have that, at any iteration $t$,
    \begin{align}
        \sum_{\tau=0}^t c^{t-\tau} a_m \left( \E[f(\bar{x}(\tau))]-f(x^\star) \right) 
        &\geq \left(a_m\sum_{\tau=0}^t c^{t-\tau} \right)\left(\min_{\tau\leq t} \E[f(\bar{x}(\tau))]-f(x^\star)\right)\nonumber\\
        &=\left(a_m\frac{1-c^{t+1}}{1-c} \right)\left(\ff(\bx(t))-f(x^\star)\right)\label{eq:xstar}
    \end{align}
    where we used Lemma~\ref{lemma:series} and the definition of $\ff$.
    Now, by making some manipulation on the term $\E[f(\bar{x}(\tau))]-f(x^\star)=\E[f(\bx(\tau))-f(x^\star)]$, as in~\cite[Theorem~1]{farina2019randomized} we get
    \begin{align}
        \E[f(\bar{x}(\tau))-f(x^\star)]\leq\sum_{i=1}^N \E[\left(f_i(y_i(\tau))-f_i(x^\star)\right)] +\sum_{i=1}^N G_i\left(\E[\|y_i(\tau)-x_i(\tau)\|]+\E[\|x_i(\tau)-\bar{x}(\tau)\|]\right).\label{eq:x_i-xstar}
    \end{align}
    In the case $c\neq \mu_M$, by substituting~\eqref{eq:x_i-xstar} in~\eqref{eq:xstar} and by using~\eqref{eq:fsum} and Lemma~\ref{lemma:x_constant} one has 
    \begin{align*}
        \left(a_m\frac{1-c^{t+1}}{1-c} \right)&(\ff(\bx(t))-f(x^\star))\leq c^{t+1} \E[V^0]  + \frac{a_M^2 \bG}{2\sigma}\frac{1-c^{t+1}}{1-c}+ 3 a_m G\left( c^{t}\left(C + \frac{\bR}{c-\mu_M}\right) +\frac{1-c^t}{1-c}\bS \right).
    \end{align*}
    Now, by dividing both sides by $a_m$ and rearranging the term one has
    \begin{align*}
        \left(\frac{1-c^{t+1}}{1-c} \right)(\ff(\bx(t))-f(x^\star))
        &\leq c^{t+1} \frac{\E[V^0]}{a_m} + c^{t} 3G\left(C + \frac{\bR}{c-\mu_M}\right)+ \frac{1-c^{t+1}}{1-c} \frac{a_M^2 \bG}{2\sigma a_m} +\frac{1-c^t}{1-c}3G\bS \nonumber\\
        & \leq c^{t} \left(\frac{\E[V^0]}{a_m} + 3GC + \frac{3G\bR}{c-\mu_M}\right)  + \frac{1-c^{t+1}}{1-c} \left(\frac{a_M^2 \bG}{2\sigma a_m} + 3G\bS\right)
    \end{align*}
    where in the second line we used the fact that $c\leq1$. Finally,~\eqref{eq:f_bound} is obtained by dividing both sides by $\frac{1-c^{t+1}}{1-c}$. The case $c=\mu_m$ can be proven in a similar way. In fact, by using the same arguments as before, we have
    \begin{align*}
        \left(\frac{1-c^{t+1}}{1-c} \right)(\ff(\bx(t))-f(x^\star))
        &\leq c^{t} \frac{\E[V^0]}{a_m} + c^{t} 3G\left(C + \frac{t\bR}{c}\right) + \frac{1-c^{t+1}}{1-c} \left(\frac{a_M^2 \bG}{2\sigma a_m} + 3G\bS\right)\nonumber\\
        & \leq c^{t} \left(\frac{\E[V^0]}{a_m} + 3GC \right) + t c^{t}\frac{3G\bR}{c}  + \frac{1-c^{t+1}}{1-c} \left(\frac{a_M^2 \bG}{2\sigma a_m} + 3G\bS\right)
    \end{align*}
    thus leading to ~\eqref{eq:f_bound_eq} by dividing both sides by $\frac{1-c^{t+1}}{1-c}$.
\end{proof}
Notice that Theorem~\ref{theorem:bound} implies that convergence with a constant error is attained, i.e., define $\tilde{f}^\star = f(x^\star) + S$, then 
\begin{equation}
    \lim_{t\to\infty} \ff(\bx^t)-\tilde{f}^\star = 0.
\end{equation}
Moreover, the convergence rate is linear. In fact, recall that $c\in(0,1)$. Then, if $c\neq \mu_M$ one has
\begin{align*}
    \lim_{t\to\infty}\frac{\ff(\bx^{t+1})-\tilde{f}^\star}{\ff(\bx^t)-\tilde{f}^\star} &\leq \lim_{t\to\infty}\frac{\frac{c^{t+1}}{1-c^{t+2}}}{\frac{c^{t}}{1-c^{t+1}}} 
    = c,
\end{align*}
while, if $c=\mu_M$,
\begin{align*}
    \lim_{t\to\infty}\frac{\ff(\bx^{t+1})-\tilde{f}^\star}{\ff(\bx^t)-\tilde{f}^\star} &\leq \lim_{t\to\infty}\frac{\frac{c^{t+1}}{1-c^{t+2}}\left(\bar{\beta} +(t+1)\eta\right)}{\frac{c^{t}}{1-c^{t+1}}\left(\bar{\beta} +t\eta\right)} 
    = c.
\end{align*}

\begin{remark}
    Our block-wise algorithm has two main benefits in terms of communication and computation respectively. First, when a limited bandwidth is available in the communication channels, data that exceed the communication bandwidth are transmitted sequentially in classical algorithms. For example, if only one block fits the communication channel, our algorithm performs an update at each communication round, while classical ones need $B$ communication rounds per update. Second, in general, solving the minimization problem in~\eqref{eq:x_update} on the entire optimization variable or on a single block results in completely different computational times. 
\end{remark}

\begin{figure}[t!]
    \centering 
    \includegraphics[width=0.5\columnwidth]{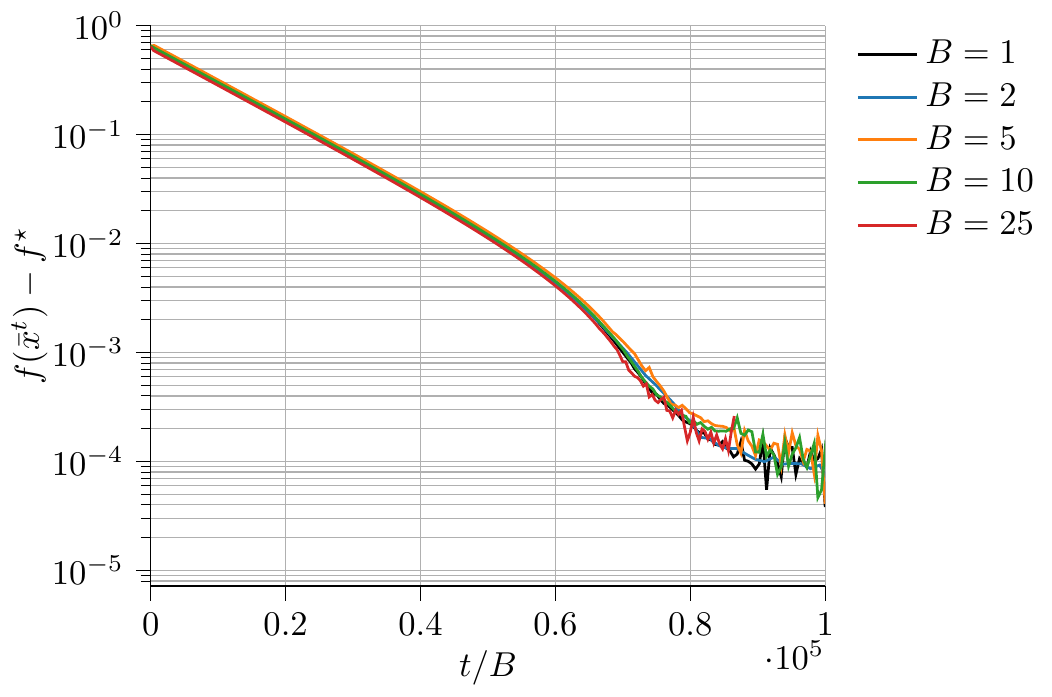}
    \caption{Numerical example: Evolution of the cost error normalized on the number of blocks.}
    \label{fig:cost}
\end{figure}

\section{Numerical example}\label{sec:experiment}
We consider as a numerical example a learning problem in which agents have to
classify samples belonging to two clusters. Formally, each agent $i\in\until{N}$
has $m_i$ training samples $q_i^1,\dots,q_{i}^{m_i}\in\R^d$ each of which has an
associated binary label $b_{i}^r\in\{-1,1\}$ for all $r\in\until{m_i}$. The goal
of the agents is to compute in a distributed way a linear classifier from the
training samples, i.e., to find a hyperplane of the form
$\{z\in\R^{d}\mid \langle \theta, z\rangle + \theta_0=0\}$, with $\theta\in\R^d$
and $\theta_0\in\R$, which better separates the training data. For notational
convenience, let $x=[\theta^\top, \theta_0]^\top\in\R^{d+1}$ and
$\hat{q}_{i}^r=[(q_{i}^r)^\top, 1]^\top$. Then, the presented problem can be
addressed by solving the following convex optimization problem, in which a
regularized Hinge loss is used as cost function,
\begin{equation*}\label{pb:regression}
    \begin{aligned}
        &\m_{x\in\R^{d+1}} 
        & & \sum_{i=1}^N\frac{1}{m_i}\sum_{r=1}^{m_i}\max\left(0, 1-b_{i}^r\langle x,\hat{q}_{i}^r\rangle\right)+\frac{\lambda}{2}\|x\|^2,
    \end{aligned}
\end{equation*}
where $\lambda>0$ is the regularization weight. This problem can be written in
the form of~\eqref{pb:problem} by defining
$\xi_i^r=(\hat{q}_{i}^r,b_i^r)$ and
\begin{align*}
    \E[h_i(x;\xi_i)]
    &=\frac{1}{m_i}\sum_{r=1}^{m_i}\left(\max\left(1-b_{i}^r\langle x,\hat{q}_{i}^r\rangle\right)+\frac{\lambda}{2N}\|x\|^2\right)
\end{align*}
for all $i\in\until{N}$.
In fact, as long as each data $\xi_i^r$ is uniformly
drawn from the dataset, Assumption~~\ref{assumption:problem_structure}\ref{subg}
is satisfied.  We implemented the algorithm in DISROPT~\cite{farina2019disropt}
and we tested it in this scenario with $N=48$ agents, $x\in\R^{50}$ and
different number of blocks, namely $B\in\{1,2,5,10,25\}$. We generated a
synthetic dataset composed of $480$ points and assigned $10$ of them to each
agent, i.e., $m_1=\dots=m_N=10$. Agents communicate according to a connected
graph generated according to an Erd\H{o}s-R\`{e}nyi random model with
connectivity parameter $p=0.5$. The corresponding weight matrix is built by
using the Metropolis-Hastings rule. Finally, we set $\lambda=1$, $p_{i,\ell}=1/B$ for all $i$ and all $\ell$, $p_{i,on}=0.95$ for all $i$ and local (constant) stepsizes $\alpha_i$ randomly chosen according to a normal distribution with  mean $0.005$ and standard deviation $10^{-4}$. The evolution of the cost error adjusted
with respect to the number of blocks is reported in Figure~\ref{fig:cost} for
the considered block numbers. The linear convergence rate can be easily
appreciated from the figure and confirms the theoretical analysis.

\section{Conclusions}\label{sec:conclusion}
In this paper, we studied the behavior of the \DBP when applied to problems
involving (non-smooth) strongly convex functions and when agents in the network
employ constant stepsizes. A linear convergence rate (with a constant error)
has been obtained in terms of the expected distance from the optimal cost. A
numerical example involving a learning problem confirmed the theoretical
analysis.

\bibliographystyle{IEEEtran}
\bibliography{biblio}

\end{document}